\documentclass[11pt]{article}
\usepackage[utf8]{inputenc}
\usepackage[
papersize={5.5in, 8.5in},
left=0.35in,
right=0.35in,
top=0.35in,
bottom=0.5in]{geometry}
\usepackage{graphicx}
\usepackage{amsthm}
\usepackage{rotating}
\usepackage{amsfonts}

\graphicspath{ {images/} }
\newtheorem*{theorem}{Theorem}
\title{\bf 
A small 6-chromatic \\ two-distance graph in the plane}
\author{\bf \normalsize Jaan Parts} 
\date{\normalsize Kazan, Russia, jaan\_parts@.mail.ru}

\begin{document}

\maketitle

\pagestyle{empty}
\thispagestyle{empty}

\begin{abstract}
We give a new, simple proof for the lower bound of the chromatic number of the Euclidean plane with two forbidden distances, based on a graph with only 16 vertices.
\end{abstract}
	
\section{Introduction}

One of the possible extensions of the classical coloring problem in which it is necessary to find the chromatic number $\chi$ of the Euclidean plane (the minimum number of colors sufficient to color the plane so that any two points unit distance apart have different colors) is to stipulate two forbidden distances $\{1, d\}$ instead of a single one. We will use the notation $\chi (d)$ for this case, implying that one of the forbidden distances is always 1. We can also stipulate that $d \ge 1$, since scaling by $1/d$ also gives the forbidden distance 1.

Since $\chi$ is known to be at least 5 \cite{grey1}, it is of interest to explore whether there is any $d$ for which $\chi (d)\ge 6$. Huddleston \cite{hud} showed that $\chi((\sqrt5+1)/2)\ge 6$. Exoo and Ismailescu \cite{exoo} gave a constructive proof for $d = 2$. P\'alv\"olgyi and \'Agoston \cite{pal} extended this list, using a probabilistic approach, to $d=\sqrt3$ and $d=(\sqrt3+1)/{\sqrt2}$.

\section{The basic construction}

In his proof, Huddleston considers a set of $5^5$ sets of five elements \cite{hud}. On this set of elements, each of which is characterized by coordinates and color, he introduces an operation of addition, which can be interpreted as obtaining a vertex colored in one of five colors. From combinatorial considerations, Huddleston concludes that vertices located at a distance $r =5$ in any 5-coloring have the same color, which implies $\chi(d)\ge 6$. However, he does not explicitly construct the graph, leaving the impression that it may be huge.

\begin{figure}[!b]
\centering
\includegraphics[scale=0.33]{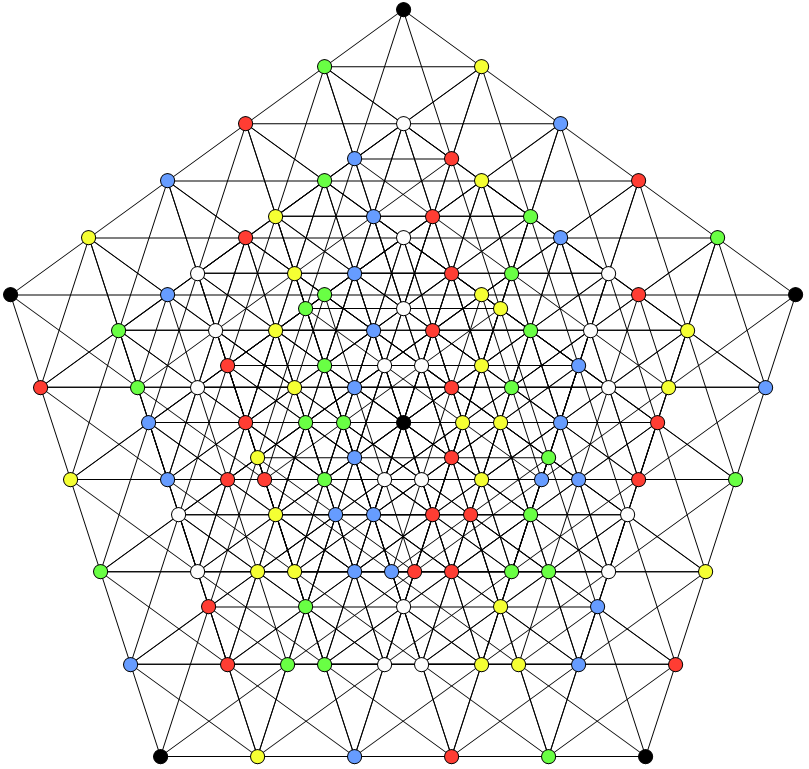}
\caption{The graph $G_{126}$ with monochromatic vertices shown in black.}
\label{g126}
\end{figure}

Corresponding graph was defined explicitly by de Grey \cite{grey}:

$$ G_{126}=G_5 \oplus G_5 \oplus G_5  \oplus G_5  \oplus G_5, $$
$$ G_5= \left\{ \left(R\sin\frac{2\pi k}{5},R\cos\frac{2\pi k}{5} \right) : R=\sqrt{\frac{5+\sqrt5}{10}},\, k\in \mathbb{Z} \right\} .$$

Here $G_5$ is a regular pentagon with side 1 and diagonal $d=(\sqrt5+1)/2$, and $\oplus$ denotes the Minkowski sum. The Minkowski sum of distance graphs is a graph whose vertex set is the union of all possible sums of their vertex coordinates, and the set of edges includes all pairs of obtained vertices for all forbidden distances.
The size of this graph is much smaller than we expected -- it has only 126~vertices, and 350~edges each of length 1 and $d$ (Fig. \ref{g126}). This is less than in $\cite{exoo}$, where the corresponding graph has 214~vertices. In any 5-coloring, the farthest vertices of the graph $G_ {126}$ from the center form monochromatic pairs with a distance $r=5$. There is also a color pattern in the form of monochrome stripes.

In addition, it turns out that the central vertex must have the same color, which gives monochromatic pairs at a distance $r=5R$. We find that this allows a large reduction in the size of the required graph. The graph $G_{16}$ that we have obtained is shown on the right in Fig. \ref{g16}. It is a subgraph of the graph $G_{126}$, has 16~vertices,
a symmetry group of order~48, and 28~edges each of length 1 and $d$, namely:
\begin{itemize}
\itemsep=0pt
    \item
$E(1)=\,$\{$\{1,2\}$, $\{1,3\}$, $\{2,4\}$, $\{2,5\}$, $\{3,4\}$, $\{3,6\}$, $\{4,7\}$, $\{4,8\}$, $\{5,6\}$, $\{5,8\}$, $\{5,9\}$, $\{6,7\}$, $\{6,10\}$, $\{7,9\}$, $\{7,12\}$, $\{7,13\}$, $\{8,10\}$, $\{8,11\}$, $\{8,13\}$, $\{9,11\}$, $\{10,12\}$, $\{11,12\}$, $\{11,14\}$, $\{12,15\}$, \\ $\{13,14\}$, $\{13,15\}$, $\{14,16\}$, $\{15,16\}$\};
    \item 
$E(d)=\;$\{$\{1,5\}$, $\{1,6\}$, $\{2,3\}$, $\{2,6\}$, $\{2,7\}$, $\{2,9\}$, $\{3,5\}$, $\{3,8\}$, $\{3,10\}$, $\{4,9\}$, $\{4,10\}$, $\{4,11\}$, $\{4,12\}$, $\{5,13\}$, $\{6,13\}$, $\{7,10\}$, \\ $\{7,14\}$, $\{8,9\}$, $\{8,15\}$, $\{9,13\}$, $\{9,14\}$, $\{10,13\}$, $\{10,15\}$, $\{11,15\}$, $\{11,16\}$, $\{12,14\}$, $\{12,16\}$, $\{14,15\}$\}.
\end{itemize}

\begin{figure}[!b]
{
\centering
\begin{tabular}{cp{0.5cm}c}
    \includegraphics[scale=0.48]{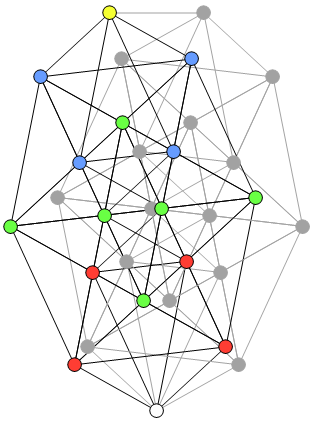} & &
    \includegraphics[scale=0.48]{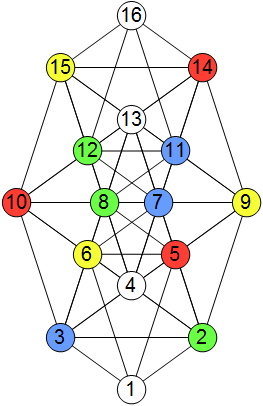} \\
    $G_{31}$ && $G_{16}$ 
\end{tabular} \par
}
    \caption{A small 6-chromatic two-distance graph, and its half-graph.}
    \label{g16}
\end{figure}

This leads to a new simple proof of the following

\newpage

\begin{theorem}
{\upshape \cite{hud}} $\chi((\sqrt5+1)/2)\ge 6$.
\end{theorem}

\begin{proof}
Consider the 31-vertex graph $G_{31}$, shown on the left in Fig. \ref{g16} and formed from two copies of $G_{16}$ rotated around a common vertex 1 by the angle $\arccos((95+\sqrt5)/100)$, which gives an additional edge of unit length between the copies of vertex 16 (alternatively, you can take the angle $\arccos((95-\sqrt5)/100)$ giving an additional edge of length $d$). This forces at least one of the two copies of vertex 16 to be a different colour than the common vertex 1. Graph $G_ {16}$ contains 5-cliques, so at least 5 colors are needed.
Thus it remains only to show that, in any 5-colouring of $G_{16}$, vertices 1 and 16 are the same colour.

1. We break the set of vertices into the following subsets: $\{1\}$, $\{2, 3, \\5, 6\}$, $\{4, 7, 8, 9, 10, 13\}$, $\{11, 12, 14, 15\}$, $\{16\}$.

2. Since the vertices $\{1, 2, 3, 5, 6\}$ form a 5-clique, we assign different colors to these vertices: 1-white, 2-green, 3-blue, 5-red, 6-yellow.

3. We notice that the set $\{4, 7, 8, 9, 10, 13\}$ has only three independent subsets (i.e., non-edges): $\{4, 13\}$, $\{7, 8\}$ and $\{9, 10\}$. We also note that the colors {green, blue, red, yellow} can be used in $\{4, 7, 8, 9, 10, 13\}$ only once. This means that one of $\{4, 13\}$, $\{7,  8\}$, $\{9, 10\}$ must have both of its vertices colored white.

4. This forbids all vertices of the set $\{11, 12, 14, 15\}$ to be white.

5. This forces vertex 16 to be white, i.e. the same colour as vertex 1.

\end{proof}

\section{A small addition}

Additionally, we identify a new graph $G_{199}$ with a monochromatic pair at a distance of $r=5/\sqrt3$ in a 5-coloring with forbidden distances $\{1, 2\}$, which leads to a 397-vertex 6-chromatic two-distance graph. The graph $G_{199}$ has 199~vertices, 870~edges of length 1 and 273 edges of length 2, a symmetry group of order 12 (see Fig.\ref{g199}). It is obtained by reducing a 313-vertex graph, which in turn is obtained from the graph $G_{19} \oplus G_{19} \oplus G_{7}$ by discarding vertices that do not fit inside the circle of radius $\sqrt3$, where $G_7$ is the 7-vertex hexagonal wheel graph and $ G_{19}=\bigcup_{k=-1}^1 G_7\,\exp(\frac{ik}{2}\arccos(\frac56))$. We also tried a monochromatic pair with a distance $r=5/3$, but it leads to a larger graph after reduction.

\begin{figure}[!t]
\centering
\includegraphics[scale=0.33]{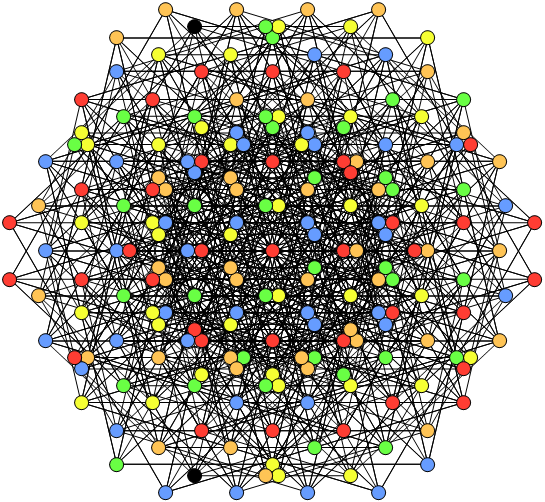}
\caption{A 5-coloring of $G_{199}$. A monochromatic pair of vertices $(-1/2, \pm5/\sqrt{12})$ is shown in black.}
\label{g199}
\end{figure}

\newpage
\section{Acknowledgements}

I thank Aubrey de Grey for cleaning up this article and clarifying the proof of Huddleston, Geoffrey Exoo and Dan Ismailescu for digging this proof, Alexander Soifer for great patience and Pink Floyd for the music in my head.

\end{document}